\documentclass[a4paper,11pt]{amsart}

\usepackage{amssymb}
\usepackage{graphicx}
\usepackage{amsmath}
\usepackage{amssymb}
\usepackage[dvips]{color}
\usepackage{booktabs}
\usepackage{bbm}

\newtheorem{priteo}{Theorem}
\newtheorem{lema}{Lemma}

\newtheorem{defi}{Definition}
\newtheorem{prop}{Proposition}

\begin{document}

\title[Heteroclinic Cycles in ODEs with the Symmetry of the Quaternionic $\mathbf{Q}_8$ Group]
{Heteroclinic Cycles in ODEs with the Symmetry of the Quaternion $\mathbf{Q}_8$ Group}

\author{Adrian C. Murza}
\address{Adrian C. Murza, Institute of Mathematics ``Simion Stoilow''
of the Romanian Academy, Calea Grivi\c tei 21, 010702 Bucharest, Romania}
\email{adrian\_murza@hotmail.com}

\begin{abstract}
In this paper we analyze the heteroclinic cycle and the Hopf bifurcation of a generic dynamical system with the symmetry of the group $\mathbf{Q}_8,$ constructed via a Cayley graph. While the Hopf bifurcation is similar to that of a $\mathbf{D}_8$--equivariant system, our main result comes from analyzing the system under weak coupling. We identify the conditions for heteroclinic cycle between three equilibria in the three--dimensional fixed point subspace of a certain isotropy subgroup of $\mathbf{Q}_8\times\mathbf{S}^1.$ We also analyze the stability of the heteroclinic cycle.
\end{abstract}



\maketitle

\section{Introduction}

Heteroclinic cycles with in systems with symmetry have been widely studied over the large decades \cite{Stork, G1, murza, podvi4, podvi2, Stork1}. During the last couple of years a special interest has received the existence of heteroclinic cycles in systems related with quaternionic symmetry, see for example the works of X. Zhang \cite{Z} and O. Podvigina \cite{podvi4, podvi2}. This is basically due to two facts. On the one hand quaternions are involved in the study of heteroclinic cycles in ODEs with symmetry in a natural way, owing to the easy representation of the dynamics in $\mathbb{R}^4$ in terms of quaternions. Many of the dynamical systems giving rise to heteroclinic cycles studied so far are $\mathbf{D}_n$--equivariant; the action of $\mathbf{D}_n$ in $\mathbb{R}^2$ is absolutely irreducible, so $\mathbb{R}^4$ is $\mathbf{D}_n$--simple. Therefore, quaternionic representations in $\mathbb{R}^4$ turned out to be very useful. On the other hand there is the intrinsic interest in the differential equations where the variables are the quaternions. We relate the study of heteroclinic cycles with the dynamics of networks of $n$ coupled oscillators with symmetry. Ashwin and Swift \cite{Ashwin_Swift} showed that the symmetry group of the network can be considered a subgroup of $\mathbb{S}_n$, as long as the oscillators taken individually have no internal symmetries. Besides these two main reasons, there are also other ones that stimulates the analysis of dynamical systems with the quaternionic symmetry, and these are related to applications to other sciences. For example we can cite the heteroclinic phenomena observed in systems with quaternionic symmetry such as nematic liquid crystals \cite{cop}, particle physics \cite{dev} and improving computational efficiency \cite{fun}; however, these heteroclinic behaviors in such systems have never been encountered a theoretical explanation. This is one of our major motivation, together with the intrinsic value of the mathematical theory developed around this subject.

An important step in designing oscillatory networks with the symmetry of a specific group has been developed by Stork \cite{Stork1}.
The authors have shown how to construct an oscillatory network with certain designed symmetry, by the Cayley graph of the symmetry group.

In this paper we analyze the heteroclinic cycles and Hopf bifurcation in ODEs with the symmetry of the quaternionic group $\mathbf{Q}_8$ of order $16.$
We use the methodology developed by Ashwin and Stork \cite{Stork} to construct a network of differential systems with $\mathbf{Q}_8$ symmetry. We investigate the dynamical behavior of the system under the weak coupling. In this case we reduce the asymptotic dynamics to a flow on an sixteen-dimensional torus $\mathbb{T}^{16}.$ We prove the existence of heteroclinic cycles between the three steady--states existing within a three--dimensional fixed--point subspace of one of the isotropy subgroups of $\mathbf{Q}_8\times\mathbf{S}^1,$ namely $\mathbf{Z}_2.$ We also classify the stability of heteroclinic cycles.

The paper is organized as follows. In Section \ref{section Cayley} we construct the most general oscillatory system with the $\mathbf{Q}_8$ symmetry by using the Cayley graph of this group. In Section \ref{section Hopf bifurcation} we analyze the Hopf bifurcation of the constructed array. In Section \ref{section Weak Coupling} we prove the existence of heteroclinic cycles in some of the subspaces which are invariant under the action of certain isotropy subgroups of $\mathbf{Q}_8.$ We also analyze their stability.

\section{The Cayley graph of the $\mathbf{Q}_8$ group}\label{section Cayley}
In this section we construct an oscillatory system with the $\mathbf{Q}_8$ symmetry and describe the elements of this group, as the relationships between them. For more details about the use of the Cayley graph in constructing the network with the prescribed symmetry see \cite{Stork1} or \cite{murza}. The Cayley graphs for $\mathbf{Q}_8$ is shown in Figure \eqref{second_figure}.

\begin{figure}[ht]
\centering
\begin{center}
\includegraphics[scale=0.25]{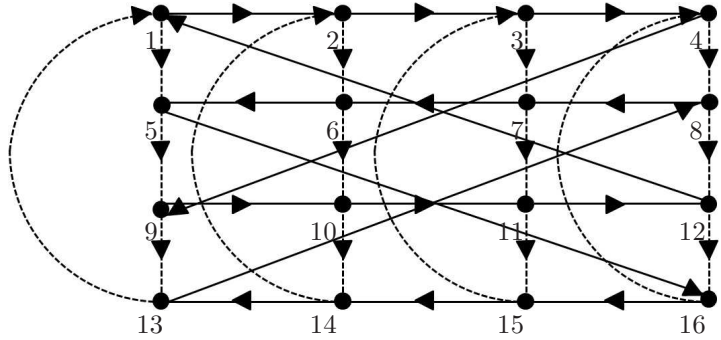}
\caption{A Cayley graph of the $\mathbf{Q}_8$ group. Solid arrows represent left-multiplication with $a,$ dot-and-dashed arrows left multiplication with $b,$ the two generators of this group.}\label{second_figure}
\begin{picture}(0,0)\vspace{0.5cm}
\put(-82,158){\small{$1$}}
\put(-14,158){\small{$2$}}
\put(55,158){\small{$3$}}
\put(122,158){\small{$4$}}
\put(-82,124){\small{$5$}}
\put(-14,124){\small{$6$}}
\put(55,124){\small{$7$}}
\put(122,124){\small{$8$}}
\put(-82,86){\small{$9$}}
\put(-20,86){\small{$10$}}
\put(50,86){\small{$11$}}
\put(118,86){\small{$12$}}
\put(-85,51){\small{$13$}}
\put(-20,51){\small{$14$}}
\put(50,51){\small{$15$}}
\put(118,51){\small{$16$}}
\end{picture}\vspace{-0.3cm}
\end{center}
\end{figure}

The action of the group $\mathbf{Q}_8$ on the cells can be written as
\begin{equation}\label{elements}
\begin{array}{l}
Id\\
a=(1~2~3~4~9~10~11~12)(5~16~15~14~13~8~7~6)\\
b=(1~5~9~13)(2~6~10~14)(3~7~11~15)(4~8~12~16)\\
ab=(1~16~9~8)(2~5~10~13)(3~6~11~14)(4~7~12~15)\\
b^2=(1~9)(2~10)(2~11)(4~12)(5~13)(6~14)(7~15)(8~16)\\
a^2=(1~3~9~11)(2~4~10~12)(5~15~13~7)(6~16~14~8)\\
a^3=(1~4~11~2~9~12~3~10)(5~14~7~16~13~6~15~8)\\
ab^2=(1~10~3~12~9~2~11~4)(5~8~15~6~13~16~7~4)\\
a^2b^2=(1~11~9~3)(2~12~10~4)(5~7~13~15)(6~8~14~16)\\
a^3b^2=(1~12~11~10~9~4~3~2)(5~6~7~8~13~14~15~16)\\
ba=(1~6~9~14)(2~7~10~15)(3~8~11~16)(4~13~12~15)\\
ba^2=(1~7~9~15)(2~8~10~16)(3~13~11~5)(4~14~12~16)\\
b^3=(1~13~9~5)(2~14)(3~15~11~7)(4~16~12~8)(6~10)\\
ab^3=(1~8~9~16)(2~13~10~5)(3~4~15~14)(6~11~12~7)\\
a^3b=(1~14~9~6)(2~15~10~7)(3~16~11~8)(4~5~12~13)\\
a^2b=(1~15~9~7)(2~16~10~8)(3~5~11~13)(4~6~12~14)
\end{array}
\end{equation}
with the relationship between them
\begin{equation}\label{commutators}
\begin{array}{l}
a^8=Id,~~~a^4=b^2=abab,~~~aba=b
\end{array}
\end{equation}

If we assign coupling $g$ between cells related by $a$ and coupling $h$ between cells related by $b$, from the permutations in \eqref{elements}, we can build the following pairwise system in   with the $\mathbf{Q}_8$ symmetry.

\begin{equation}\label{array 4 eq}
\begin{array}{l}
\dot{x}_1=f(x_1)+g(x_{12},x_1)+h(x_5,x_9),\hspace{1.47cm}\dot{x}_2=f(x_2)+g(x_1,x_2)+h(x_9,x_{13}),\\
\dot{x}_3=f(x_3)+g(x_2,x_3)+h(x_{13},x_1),\hspace{1.47cm}\dot{x}_4=f(x_4)+g(x_3,x_4)+h(x_1,x_5),\\
\dot{x}_5=f(x_5)+g(x_4,x_9)+h(x_2,x_6),\hspace{1.62cm}\dot{x}_6=f(x_6)+g(x_9,x_{10})+h(x_6,x_{10}),\\
\dot{x}_7=f(x_7)+g(x_{10},x_{11})+h(x_{10},x_{14}),\hspace{1.02cm}\dot{x}_8=f(x_8)+g(x_{11},x_{12})+h(x_{14},x_2),\\
\dot{x}_9=f(x_9)+g(x_6,x_5)+h(x_3,x_7),\hspace{1.61cm}\dot{x}_{10}=f(x_{10})+g(x_5,x_{16})+h(x_7,x_{11}),\\
\dot{x}_{11}=f(x_{11})+g(x_{16},x_{15})+h(x_{11},x_{15}),\hspace{0.7cm}\dot{x}_{12}=f(x_{12})+g(x_{15},x_{14})+h(x_{15},x_3),\\
\dot{x}_{13}=f(x_{13})+g(x_{14},x_{13})+h(x_4,x_8),\hspace{1.01cm}\dot{x}_{14}=f(x_{14})+g(x_{13},x_8)+h(x_8,x_{12}),\\
\dot{x}_{15}=f(x_{15})+g(x_8,x_7)+h(x_{12},x_{16}),\hspace{1.01cm}\dot{x}_{16}=f(x_{16})+g(x_7,x_6)+h(x_{16},x_4),
\end{array}
\end{equation}
where $f:\mathbb{R}\rightarrow\mathbb{R}$ and $g,~h:\mathbb{R}^2\rightarrow\mathbb{R}$. As shown by Ashwin and Stork \cite{Stork} we can think of $f,~g,~h$ as being generic functions that assure that the isotropy of this vector field under the action of $\mathbf{O}_{16}$ is generically  $\mathbf{Q}_8$.

\section{Hopf bifurcation}\label{section Hopf bifurcation}
In order to consider generic one-parameter Hopf bifurcation in systems with $\mathbf{Q}_8,$ we need to analyze the complex irreducible representations of $\mathbf{Q}_8.$ Based on the work of Golubitsky and Stewart \cite{GS88}, these representations are of one or two dimensions. From their theory, we have that the linear representation of a group $\Gamma$
$$\alpha_{\Gamma}:\Gamma\times W\rightarrow W$$
on the complex vector space $W$ is irreducible if and only if $\Gamma-$invariant subspaces are trivial; it is to say, $\{0\}$ or $W$ itself. It is important to notice, that (a) there need be no faithful irreducible representations, and (b) this is typical.In addition, the amount by which the representation fails to be faithful is the kernel of the action $\alpha_{\Gamma}.$\\

The group $\mathbf{Q}_8$ has five irreducible representations; four of them are one-dimensional and the remaining one is two-dimensional. The one-dimensional representations can be interpreted as Hopf bifurcation with trivial or $\mathbf{Z}_2$ symmetry, which correspond to a quotient group of $\mathbf{Q}_8.$\\

From \cite{J} the two generators of $\mathbf{Q}_8$ are
\begin{equation}\label{ab}
a=
\left(
\begin{array}{cc}
\omega&0\\
0&\bar{\omega}
\end{array}
\right),~~~b=
\left(
\begin{array}{cc}
0&-1\\
1&0
\end{array}
\right),
\end{equation}
where $\omega=\displaystyle{\exp\left(\frac{\pi i}{4}\right)}.$

Therefore the standard irreducible action of $\mathbf{Q}_8$ on $\mathbb{C}^2$ is given by
\begin{equation}\label{actions1}
\begin{array}{l}
a(z_+,z_-)=\displaystyle{\left(  \frac{\sqrt{2}}{2}(1+i)z_+,\frac{\sqrt{2}}{2}(1-i)z_-\right)}\\
\\
b(z_+,z_-)=\displaystyle{\left(  -z_-,z_+\right)}\\
\end{array}
\end{equation}
and there is a phase shift action of $\mathbf{S}^1$ given by
$$R_{\phi}(z_+,z_-)=(e^{i\phi}z_+,e^{i\phi}z_-),$$
for $\phi\in\mathbf{S}^1.$ The action of $\mathbf{Q}_8\times\mathbf{S}^1$ is similar to the action of $\mathbf{D}_8\times\mathbf{S}^1.$ This action is generated by
\begin{equation}\label{actions2}
\begin{array}{l}
\kappa(z_+,z_-)=(z_-,z_+),~\qquad
\rho(z_+,z_-)=(iz_+,-iz_-),
\end{array}
\end{equation}
where $\rho^4=\kappa^2=1$ and $\rho\kappa=\kappa\rho^3.$ In our case of the group $\mathbf{Q}_8$ we have $a^8=b^4=1,a^4=b^2$ and $aba=b.$ The kernel of this action in the $2-$cycle in $\mathbf{D}_8\times\mathbf{S}^1$ is generated by $(\rho^2,\pi)$, while the kernel of the action of $\mathbf{Q}_8\times\mathbf{S}^1$ is generated by $(a^4=b^2=1,\pi)$. Therefore, it is possible to check that
$$\mathbf{Q}_8\times\mathbf{S}^1/ker\alpha_{\mathbf{Q}_8\times\mathbf{S}^1}\equiv\mathbf{Q}_8\times\mathbf{S}^1/
ker\alpha_{\mathbf{D}_8\times\mathbf{S}^1}.$$
This means that we use the results obtained in \cite{GS88} for Hopf bifurcation in systems with $\mathbf{D}_8$ symmetry, with a re-interpretation of the branches.

\begin{table}
\centering
\begin{center}
\caption{Table that relates the isotropies of points in $\mathbb{C}^2$ for identical actions of $\mathbf{Q}_8\times\mathbf{S}^1$ and $\mathbf{D}_8\times\mathbf{S}^1$.}\label{auctions}
\end{center}
\begin{tabular}{c|c|c|c|c}
\toprule
Isotropy in $\mathbf{Q}_8\times\mathbf{S}^1$ & Isotropy in $\mathbf{D}_8\times\mathbf{S}^1$& Fix&$dim_{\mathbb{C}}Fix$&Name ($\mathbf{D}_8$) \\
\midrule
\centering
$\mathbf{Q}_8\times\mathbf{S}^1$ &$\mathbf{D}_8\times\mathbf{S}^1$&$(0,0)$&$0$&Trivial solution\\
$\tilde{\mathbf{Z}}_8^a$&$\tilde{\mathbf{Z}}_8(\rho)$&$(z,0)$&1&Rotating Wave\\
$\tilde{\mathbf{Z}}_8^b$&$\tilde{\mathbf{Z}}_2(\rho^2)\times\tilde{\mathbf{Z}}_2(\kappa)$&$(z,z)$&$1$&Edge Solution\\
$\tilde{\mathbf{Z}}_8^c$&$\tilde{\mathbf{Z}}_2(\rho^2)\times\tilde{\mathbf{Z}}_2(\rho\kappa)$&$(z,iz)$&$1$&Vertex Oscillation\\
$\tilde{\mathbf{Z}}_2$&$\tilde{\mathbf{Z}}_2(\rho^2)$&$(w,z)$&$2$&Submaximal\\
\bottomrule
\end{tabular}
\end{table}

\vspace{0.3cm}
\begin{prop}
There are exactly three branches of periodic solutions that bifurcate from $(0,0),$ corresponding to the isotropy subgroups $\mathbf{D}_8\times\mathbf{S}^1$ with two-dimensional fixed-point subspaces.
\end{prop}
\begin{proof}
The proof is a direct application to the group $\mathbf{D}_8$ of Theorem $4.2$ of \cite{paiva}. Therefore, there are exactly three branches of periodic solutions occurring generically in Hopf bifurcation with $\mathbf{D}_8$ symmetry.
\end{proof}

From Proposition $2.1,$ page $372$ in \cite{GS88} we have that every smooth $\mathbf{D}_8\times\mathbf{S}^1-$equivariant map germ $g:\mathbb{C}^2\rightarrow\mathbb{C}^2$ has the form
\begin{equation}\label{GS1}
\begin{array}{l}
g(z_1,z_2)=A\begin{bmatrix}z_1\\z_2\end{bmatrix}+B\begin{bmatrix}z_1^2\bar{z}_1\\z_2^2\bar{z}_2\end{bmatrix}+
C\begin{bmatrix}\bar{z}_1^3z_2^4\\z_1^4\bar{z}_2^3\end{bmatrix}+D\begin{bmatrix}z_1^5\bar{z}_2^4\\ \bar{z}_1^4z_2^5\end{bmatrix},
\end{array}
\end{equation}
where A, B, C, D are complex-valued $\mathbf{D}_8\times\mathbf{S}^1-$invariant functions.
The branching equations for $\mathbf{D}_8-$equivariant Hopf bifurcation may be rewritten $g(z_1,z_2)=0.$
These branching equations are shown in Table \eqref{branches}.
\begin{table}
\centering
\begin{center}
\caption{Branching Equations for $\mathbf{D}_8$ Hopf Bifurcation.}\label{branches}
\end{center}
\begin{tabular}{c|c|c}
\toprule
Orbit type& Branching Equations& Signs of Eigenvalues \\
\midrule
\centering
$(0,0)$&-&$\mathrm{Re}A$(0,$\lambda)$\\
$(a,0)$&$A+Ba^2=0$&$\mathrm{Re}(A_N+B)+O(a)$\\
&&$-\mathrm{Re}(B)$ [twice]\\
$(a,a)$&$A+Ba^2+Ca^6+Da^8=0$&$\mathrm{Re}(2A_N+B)+O(a)$\\
&&$
\left\{
\begin{array}{l}
\mathrm{trace}=\mathrm{Re}(B)+O(a)\\
\mathrm{det}=-\mathrm{Re}(B\bar{C})+O(a)\\
\end{array}
\right.$\\
$(a,e^{\pi i/4}a)$&$A+Ba^2-Ca^6-Da^8=0$&$\mathrm{Re}(2A_N+B)+O(a)$\\
&&$
\left\{
\begin{array}{l}
\mathrm{trace}=\mathrm{Re}(B)+O(a)\\
\mathrm{det}=\mathrm{Re}(B\bar{C})+O(a)\\
\end{array}
\right.$\\
\bottomrule
\end{tabular}
\end{table}
\subsection{Bifurcating branches}
We now use the information in Table \eqref{branches} to derive the bifurcation diagrams describing the generic $\mathbf{D}_8-$equivariant Hopf bifurcation. Assume
\small
\begin{equation}\label{brancing2}
\begin{array}{l}
(a)~~\mathrm{Re}(A_N+B)\neq0,~~(b)~~\mathrm{Re}(B)\neq0,~~(c)~~\mathrm{Re}(2A_N+B)\neq0,~~(d)~~\mathrm{Re}(B\bar{C})\neq0,~~
(c)~~\mathrm{Re}(A_{\lambda})\neq0,
\end{array}
\end{equation}
\normalsize
where each term is evaluated at the origin.\\
Assuming nondegeneracy conditions \eqref{brancing2} and the trivial branch is stable subcritically and loses stability as bifurcation parameter $\lambda$ passes through $0.$ We summarize these facts into the next theorem.
\begin{priteo}\label{teor stability}
The following statements hold.
\begin{itemize}
\item [(a)] The $\tilde{\mathbf{Z}}_8$ branch is super- or subcritical according to whether $\mathrm{Re}(A_N(0)+B(0))$ is positive or negative. It is stable if $\mathrm{Re}(A_N(0)+B(0))>0,~\mathrm{Re}(B(0))<0.$
\item [(b)] The $\mathbf{Z}_2(\kappa)[\oplus\mathbf{Z}_2^c]$ is super- or subcritical according to wether $\mathrm{Re}(2A_N(0)+B(0))$ is positive or negative. It is stable if $\mathrm{Re}(2A_N(0)+B(0))>0,~\mathrm{Re}(B(0))>0$ and $\mathrm{Re}(2B(0)\bar{C}(0))<0.$
\item [(c)] The $\mathbf{Z}_2(\kappa,\pi)[\oplus\mathbf{Z}_2^c]$ or $\mathbf{Z}_2(\kappa,\xi)\oplus\mathbf{Z}_2^c$ branch is super- or subcritical according to wether $\mathrm{Re}(2A_N(0)+B(0))$ is positive or negative. It is stable if $\mathrm{Re}(2A_N(0)+B(0))>0,~\mathrm{Re}(B(0))>0$ and $\mathrm{Re}(2B(0)\bar{C}(0))<0.$
\end{itemize}
\end{priteo}
\begin{proof}
The proof is a direct application to the case $\mathbf{D}_8$ of the Theorem $3.1$ page $382$ in \cite{GS88}.
\end{proof}

\section{Weak Coupling}\label{section Weak Coupling}

The idea of studying ODEs in the weak coupling limit was introduced by Ashwin and Swift \cite{Ashwin_Swift}. This situation can be uderstood as follows. In the no coupling case there is an attracting $n$--dimensional torus with one angle for every oscillator. The situation is completely different to the Hopf bifurcation. Instead of examining small amplitude oscillations near a Hopf bifurcation point, we make a weak coupling approximation. There is a slow evolution of the phase differences in the weak coupling. Another improvement with respect to the Hopf bifurcation is that while the Hopf bifurcation theory gives local information, the weak coupling case the yields global results on the n--dimensional torus.

System \eqref{array 4 eq} can be rewritten under weak coupling case as an ODE of the form:
\begin{equation}\label{generic weak coupling equation}
\dot{x}_i=f(x_i)+\epsilon g_i(x_1,\ldots,x_{16})
\end{equation}
for $i=1,\ldots,16,~x_i\in \mathcal{V}$ and commuting with the permutation action of $\mathbf{Q}_8$ on $\mathcal{V}^{16},$ both $f$ and $g_i$ being of the class $\mathcal{C}^{\infty}.$ The constant $\epsilon$ represents the coupling strength and we have $\epsilon\ll1$. As in \cite{Ashwin_Swift}, or \cite{Stork} we may assume $\dot{x}=f(x)$ has an hyperbolic stable limit cycle.\\

It follows that if the coupling is weak, we should not just take into account the irreducible representations of $\mathbf{Q}_8.$ Since there are $16$ stable hyperbolic limit cycles in the limit of $\epsilon=0,$ it means that the asymptotic dynamics of the system factors into the asymptotic dynamics of $16$ limit cycles. We assume that each limit cycle taken individually is hyperbolic for small enough values of the coupling parameter. This justifies expressing the dynamics of the system only in terms of phases, i.e. an ODE on $\mathbf{T}^{16}$  which is $\mathbf{Q}_8-$equivariant.

\tiny
\begin{table}
\centering
\begin{center}
\caption{Isotropy subgroups and fixed point subspaces for the $\mathbf{Q}_8\times\mathbf{S}^1$ action on $\mathbf{T}^{16}$.}\label{table grande}
\end{center}
\begin{tabular}{ccccc}
\toprule
$\Sigma$ & $\mathrm{Fix}(\Sigma)$ & Generators & $\mathrm{dim~Fix}(\Sigma)$\\
\midrule
\centering
$\mathbf{Q}_8$ & $(0,0,0,0,0,0,0,0,0,0,0,0,0,0,0,0)$ & $(a,0),(b,0)$ &$0$\\
$\tilde{\mathbf{Q}}_8^a$ & $(0,0,0,0,0,0,0,0,\pi,\pi,\pi,\pi,\pi,\pi,\pi,\pi)$ & $(b,\pi),(ab,\pi)$ &$0$\\
$\tilde{\mathbf{Q}}_8^{b}$ & $(0,\pi,0,\pi,0,\pi,0,\pi,0,\pi,0,\pi,0,\pi,0,\pi)$ & $(a,\pi),(ab,\pi)$ &$0$\\
$\tilde{\mathbf{Q}}_8^{ab}$ & $(0,\pi,0,\pi,0,\pi,0,\pi,\pi,0,\pi,0,\pi,0,\pi,0)$ & $(a,\pi),(b,\pi)$ &$0$\\
$\mathbf{Z}_8^a$ & $(0,0,0,0,0,0,0,0,\phi,\phi,\phi,\phi,\phi,\phi,\phi,\phi)$ & $(a,0)$ &$1$\\
$\mathbf{Z}_8^{b}$ & $(0,\phi,0,\phi,0,\phi,0,\phi,0,\phi,0,\phi,0,\phi,0,\phi)$ & $(b,0)$ &$1$\\
$\mathbf{Z}_8^{ab}$ & $(0,\phi,0,\phi,0,\phi,0,\phi,\phi,0,\phi,0,\phi,0,\phi,0)$ & $(ab,0)$ &$1$\\
$\tilde{\mathbf{Z}}_8^a$ & $(0,\pi,0,\pi,0,\pi,0,\pi,\phi,\phi+\pi,\phi,\phi+\pi,\phi,\phi+\pi,\phi,\phi+\pi)$ & $(a,\pi)$ &1\\
$\tilde{\mathbf{Z}}_8^b$ & $(0,\phi,0,\phi,0,\phi,0,\phi,\phi+\pi,\phi,\phi+\pi,\phi,\phi+\pi,\phi,\phi+\pi,\pi)$ & $(b,\pi)$ &$1$\\
$\tilde{\mathbf{Z}}_8^{ab}$ & $(0,\phi,0,\phi,0,\phi,0,\phi+\pi,\phi,\phi+\pi,\phi,\phi+\pi,\phi,\phi+\pi,\pi,\phi+\pi)$ & $(ab,\pi)$ &1\\
$\tilde{\mathbf{Z}}_8^{a/4}$ & $(0,\frac{7\pi}{4},\frac{3\pi}{2},\frac{5\pi}{4},\pi,\frac{3\pi}{4},\frac{\pi}{2},\frac{\pi}{4},\phi,\phi+\frac{\pi}{4},
\phi+\frac{\pi}{2},\phi+\frac{3\pi}{4},\phi+\pi,\phi+\frac{5\pi}{4},\phi+\frac{3\pi}{2},\phi+\frac{7\pi}{4})$ & $(a,\frac{\pi}{4})$ &$1$\\
$\tilde{\mathbf{Z}}_8^{b/4}$&$(0,\phi,\pi,\phi+\pi,\phi+\frac{7\pi}{4},\frac{7\pi}{4},\phi+\frac{3\pi}{2},\frac{3\pi}{2},\phi+\frac{5\pi}{4},\phi,
\phi+\frac{3\pi}{4},\phi,\phi+\frac{\pi}{2},\phi,\phi+\frac{\pi}{4},\frac{\pi}{4})$& $(b,\frac{\pi}{4})$ &$1$\\
$\tilde{\mathbf{Z}}_8^{b/4}$ &$(0,\phi,\pi,\phi+\pi,\frac{\pi}{4},\phi+\frac{\pi}{4},\frac{\pi}{2},\phi+\frac{\pi}{2},\frac{3\pi}{4},\phi+
\frac{3\pi}{4},\frac{3\pi}{2},\phi+\frac{3\pi}{2},\frac{5\pi}{4},\phi+\frac{5\pi}{4},\frac{7\pi}{4},\phi+\frac{7\pi}{4})$& $(ab,\frac{\pi}{4})$ &$1$\\
$\mathbf{Z}_2$&$(0,\phi_1,0,\phi_1,0,\phi_1,0,\phi_1,\phi_2,\phi_3,\phi_2,\phi_3,\phi_2,\phi_3,\phi_2,\phi_3)$& $(b^2,0)$ &$3$\\
$\tilde{\mathbf{Z}}_2$&$(0,\phi_1,\pi,\phi_1+\pi,0,\phi_1,\pi,\phi_1+\pi,\phi_2,\phi_3,\phi_2+\pi,\phi_3+\pi,\phi_2,\phi_3,\phi_2+\pi,\phi_3+\pi)$& $(b^2,\pi)$ &$3$\\
\bottomrule
\end{tabular}
\end{table}
\normalsize

When considering the weakly coupled system we can average it over the phases \cite{Stork1}. This is the same as introducing and phase shift symmetry by translation along the diagonal;
$$R_{\theta}(\phi_1,\ldots,\phi_{16}):=(\phi_1+\theta,\ldots,\phi_{16}+\theta),$$
for $\theta\in\mathbf{S}^1.$

We obtained an ODE on that is equivariant under the action of $\mathbf{Q}_{8}\times\mathbf{S}^1,$  and we have to classify the isotropy types of points under this action. This is done in Table \eqref{table grande}.
Since now on, our interest focuses in the three-dimensional space $\mathrm{Fix}(\mathbf{Z}_2);$ it does not contain two-dimensional fixed-point subspaces. In turn, it contains several one- and zero-dimensional subspaces fixed by the isotropy subgroups $\tilde{\mathbf{Q}}_8^i,$ $\tilde{\mathbf{Z}}_8^i$ and $\mathbf{Z}_8^i,$ respectively, where $i=\{a,b,ab,a/4,b/4\}$ as in Table \eqref{table grande}. These symmetries are not in $\mathbf{Z}_2;$ however, they are in the normalizer of $\mathbf{Z}_2$.

\subsection{Dynamics of the $\theta_1,\theta_2$ and $\theta_3$ angles in $\mathrm{Fix}(\mathbf{Z}_2)$}\label{onetorus_theor}
We can define coordinates in $\mathrm{Fix}(\mathbf{Z}_2) $ by taking a basis
\begin{equation}\label{basis}
\begin{array}{l}
e_1=-\frac{1}{8}(1,1,1,1,1,1,1,1,-1,-1,-1,-1,-1,-1,-1,-1)\\
e_2=-\frac{1}{8}(1, -1, 1, -1 ,1, -1, 1, -1 ,1, -1, 1, -1, 1, -1, 1, -1)\\
e_3=-\frac{1}{8}(1, -1, 1, -1, 1, -1, 1, -1, -1, 1, -1, 1, -1, 1, -1, 1)
\end{array}
\end{equation}
and consider the space spanned by $\{e_1,e_2,e_3\}$ parametrized by $\{\theta_1,\theta_2,\theta_3\}:~\sum_{n=1}^3\theta_ne_n.$

By using these coordinates, we construct the following family of three-dimensional differential systems which satisfies the symmetry of $\mathrm{Fix}(\mathbf{Z}_2)$.
\begin{equation}\label{systema ejemplo}
\left\{
\begin{array}{l}
\dot{\theta_1}=u\sin{\theta_1}\cos{\theta_2}+\epsilon\sin{2\theta_1}\cos{2\theta_2}\\
\\
\dot{\theta_2}=u\sin{\theta_2}\cos{\theta_3}+\epsilon\sin{2\theta_2}\cos{2\theta_3}\\
\\
\dot{\theta_2}=u\sin{\theta_3}\cos{\theta_1}+\epsilon\sin{2\theta_3}\cos{2\theta_1}
+q(1-\cos\theta_1)\sin2\theta_3,\\
\end{array}
\right.
\end{equation}
where $u,\epsilon,q\in\mathbb{R}.$

We will show that this vector field contains structurally stable, attracting heteroclinic cycles which may be asymptotically stable, essentially asymptotically stable or completely unstable, depending on the values of $u,\epsilon$ and $q.$ We can assume, without loss of genericity that the space $\mathrm{Fix}(\mathbf{Z}_2)$ is normally attracting for the dynamics and therefore the dynamics within the fixed-point space determines the stability of the full system.
In the following we will show that the planes $\theta_i=0~(\mathrm{mod}~\pi),i=1,2,3$ are invariant under the flow of \eqref{systema ejemplo}.

Let $\mathcal{X}$ be the vector field of system \eqref{systema ejemplo}.
\begin{defi}
We call a trigonometric invariant algebraic surface $h(\theta_1,\theta_2,\theta_3)=0,$ if it is invariant by the flow of \eqref{systema ejemplo}, i.e. there exists a function $K(\theta_1,\theta_2,\theta_3)$ such that
\begin{equation}\label{campo}
\mathcal{X}h=\frac{\partial h}{\partial\theta_1}\dot{\theta_1}+\frac{\partial h}{\partial\theta_2}\dot{\theta_2}
+\frac{\partial h}{\partial\theta_3}\dot{\theta_3}=Kh.
\end{equation}
\end{defi}

\begin{lema}
Functions $\sin\theta_1,~\sin\theta_2$ and $\sin\theta_3$ are trigonometric invariant algebraic surfaces for system \eqref{systema ejemplo}.
\end{lema}
\begin{proof}
We can write the system \eqref{systema ejemplo} in the form
\begin{equation}\label{systema ejemplo1}
\left\{
\begin{array}{l}
\dot{\theta_1}=\sin{\theta_1}\left(u\cos{\theta_2}+2\epsilon\cos{\theta_1}\cos{2\theta_2}\right)\\
\\
\dot{\theta_2}=\sin{\theta_2}\left(u\cos{\theta_3}+2\epsilon\cos{\theta_2}\cos{2\theta_3}\right)\\
\\
\dot{\theta_3}=\sin{\theta_3}\left(u\cos{\theta_1}+2\epsilon\cos{2\theta_1}\cos{\theta_3}
+2q(1-\cos\theta_1)\cos\theta_3\right)\\
\end{array}
\right.
\end{equation}

Now if we choose $h_1=\sin\theta_1,$ then $\mathcal{X}h_1=\cos{\theta_1}\sin{\theta_1}\left(u\cos{\theta_2}+2\epsilon\cos{\theta_1}\cos{2\theta_2}\right)$
so $K_1=\cos{\theta_1}\left(u\cos{\theta_2}+2\epsilon\cos{\theta_1}\cos{2\theta_2}\right).$ The remaining cases follow similarly.
\end{proof}

Since the planes $\theta_i=0(\mathrm{mod}~\pi)$ are invariant under the flow of \eqref{systema ejemplo}, it is clear that $(0,0,0),$ $(\pi,0,0),$ $(0,\pi,0)$, and $(0,0,\pi)$ are equilibria for \eqref{systema ejemplo}. To check the possibility of heteroclinic cycles in system \eqref{systema ejemplo}, we linearize about the equilibria (i.e. the zero-dimensional fixed points). The idea is proving that there are three-dimensional fixed-point spaces $\mathrm{Fix}(\mathbf{Z}_2)$ and $\mathrm{Fix}(\tilde{\mathbf{Z}}_2)$ which connect these fixed points, allowing the existence of such a heteroclinic network between the equilibria.\\
Let's assume
\begin{equation}\label{param}
\begin{array}{l}
|\epsilon|<\frac{u}{2}\qquad\mathrm{and}\qquad
|\epsilon+2q|<\frac{u}{2}.
\end{array}
\end{equation}

\begin{table}
\centering
\begin{center}
\caption{Eigenvalues of the flow of equation \eqref{systema ejemplo}, at the four non-conjugate zero-dimensional fixed points.}\label{table fixed points}
\end{center}
\begin{tabular}{c|c|c|c|c}
\toprule
$\Sigma$ & $\mathrm{Fix}(\Sigma)$ with coordinates $(\phi_1,\phi_2,\phi_3)$ & $\lambda_1$ & $\lambda_2$& $\lambda_3$\\
\midrule
\centering
$\mathbf{Q}_8$& $(0,0,0)$ & $u+2\epsilon$ &$u+2\epsilon$&$u+2\epsilon$\\
\midrule
$\tilde{\mathbf{Q}}_8^a$ & $(\pi,0,0)$ & $-u+2\epsilon$ &$u+2\epsilon$&$-u+2\epsilon+4q$\\
$\tilde{\mathbf{Q}}_8^b$ & $(0,\pi,0)$ & $-u+2\epsilon$ &$-u+2\epsilon$&$u+2\epsilon$\\
$\tilde{\mathbf{Q}}_8^{ab}$ & $(0,0,\pi)$ & $u+2\epsilon$ &$-u+2\epsilon$&$-u+2\epsilon$\\
\bottomrule
\end{tabular}
\end{table}

We use the criteria of Krupa and Melbourne \cite{Krupa} to study the stability of the heteroclinic cycle.
\begin{priteo}
In the following we will prove that there exists the possibility of a heteroclinic cycle in the following way:
\begin{equation}\label{flechas}
\begin{array}{l}
\cdots\xrightarrow{\mathrm{Fix}(\tilde{\mathbf{Z}}_8^{b})}\mathrm{Fix}(\tilde{\mathbf{Q}}_8^a)\xrightarrow{\mathrm{Fix}(\tilde{\mathbf{Z}}_8^{ab})}
\mathrm{Fix}(\tilde{\mathbf{Q}}_8^b)\xrightarrow{\mathrm{Fix}(\tilde{\mathbf{Z}}_8^a)}
\mathrm{Fix}(\tilde{\mathbf{Q}}_8^{ab})
\xrightarrow{\mathrm{Fix}(\tilde{\mathbf{Z}}_8^{b})}\cdots
\end{array}
\end{equation}
The stability of the heteroclinic cycle is:
\begin{itemize}
\item [(a)] asymptotically stable if
\begin{equation}\label{eq1 teorema estabilidad}
\begin{array}{l}
\displaystyle{u<0~\mathrm{and}~q<\frac{3u}{4}-\frac{\epsilon}{2}},
\end{array}
\end{equation}
\item [(b)] unstable but essentially asymptotically stable if
\begin{equation}\label{eq1 teorema estabilidad repetition}
\begin{array}{l}
\displaystyle{u<0~\mathrm{and}~\frac{3u}{4}-\frac{\epsilon}{2}<q<\frac{u}{2}-\frac{(u+2\epsilon)^3}{(-u+2\epsilon)^2}}.
\end{array}
\end{equation}
\item [(c)] completely unstable if $u>0.$
\end{itemize}
\end{priteo}
\begin{proof}

The stability is expressed by
\begin{equation}\label{stability krupa1}
\rho=\prod_{i=1}^3\rho_i,~~~\mathrm{where}~~~\rho_i=\mathrm{min}\{c_i/e_i,1-t_i/e_i\}.
\end{equation}

In equation \eqref{stability krupa1}, $e_i$ is the expanding eigenvalue at the $i$th point of the cycle, $-c_i$ is the contracting eigenvalue and $t_i$ is the transverse eigenvalue of the linearization.
For the heteroclinic cycle we have
\begin{equation}\label{rho values11}
\rho_1=
\left\{
\begin{array}{l}
\displaystyle{\frac{2u-4q}{u+2\epsilon}}~\mathrm{if}~q<\frac{3u}{4}-\frac{\epsilon}{2},\\
\\
\displaystyle{\frac{-u+2\epsilon}{u+2\epsilon}}~\mathrm{if}~q>\frac{3u}{4}-\frac{\epsilon}{2},
\end{array}
\qquad\qquad\rho_2=\rho_3=\displaystyle{\frac{-u+2\epsilon}{u+2\epsilon}},
\right.
\end{equation}

so from equations \eqref{stability krupa1} and \eqref{rho values11} we obtain
\begin{equation}\label{rho values13}
\rho=
\left\{
\begin{array}{l}
\displaystyle{\frac{(-u+2\epsilon)^2(2u-4q)}{(u+2\epsilon)^3}}~\mathrm{if}~u<0~\mathrm{and}~q<\frac{3u}{4}-\frac{\epsilon}{2},\\
\\
\displaystyle{\frac{(-u+2\epsilon)^3}{(u+2\epsilon)^3}}~\mathrm{if}~u<0~\mathrm{and}~q>\frac{3u}{4}-\frac{\epsilon}{2}.
\end{array}
\right.
\end{equation}
Then the proof follows by applying Theorem $2.4$ in \cite{Krupa}.
\end{proof}

For any $u<0$ we have $\displaystyle{\frac{3u}{4}-\frac{\epsilon}{2}<q<\frac{u}{2}-\frac{(u+2\epsilon)^3}{(-u+2\epsilon)^2}}$ and therefore there exist values of $q$ for which there exist essentially asymptotic stable heteroclinic connections. In consequence, there exists an attracting heteroclinic cycle even though the linear stability of $\mathrm{Fix}(\tilde{\mathbf{Q}}_8^a)$ has an expanding transverse eigenvalue.

\section{Conclusions}
We prove the existence of stable heteroclinic cycles in the most general coupled ordinary differential equations with quaternionic symmetry $\mathbf{Q}_8$. Our approach is generic and offers for the first time as far as we know, evidence of these phenomena in systems with this symmetry. While the results stands on its own from a mathematical point of view, it might also contribute to a better understanding of these intermittent behaviors experimentally observed in nematic liquid crystals \cite{cop} and particle physics \cite{dev}.

\section*{Acknowledgements}
In first place I would like to address many thanks to the Referee, whose helpful indications and comments greatly improved the presentation of the paper. I acknowledge a BITDEFENDER
postdoctoral fellowship from the Institute of Mathematics "Simion
Stoilow" of the Romanian Academy, Contract of Sponsorship No.
262/2016 as well as economical support from a grant of the Romanian
National Authority for Scientific Research and Innovation,
CNCS-UEFISCDI, project number PN-II-RU-TE-2014-4-0657.

\end{document}